\documentclass[reqno]{amsart}

\usepackage{amsmath,amssymb,amscd,accents} \usepackage{graphicx}
\DeclareGraphicsExtensions{.eps} \usepackage{mathrsfs}
\usepackage[mathcal]{eucal} \usepackage{esint} \usepackage{color}
\RequirePackage[colorlinks=true, breaklinks=true, urlcolor= black, linkcolor= black, citecolor= black, bookmarksopen=true,linktocpage=true,plainpages=false,pdfpagelabels,unicode]{hyperref}

\newtheorem{Thm}{Theorem}{\bfseries}{\itshape}
\newtheorem*{Thm*}{Theorem}{\bfseries}{\itshape}
\newtheorem{Cor}{Corollary}{\bfseries}{\itshape}
\newtheorem{Prop}[Cor]{Proposition}{\bfseries}{\itshape}
\newtheorem{Lem}[Cor]{Lemma}{\bfseries}{\itshape}
\newtheorem*{Lem*}{Lemma}{\bfseries}{\itshape}
{\bfseries}{\itshape}
{\bfseries}{\itshape}
{\bfseries}{\rmfamily}
{\scshape}{\rmfamily}
\newtheorem{Rem}[Cor]{Remark}{\scshape}{\rmfamily}
{\scshape}{\rmfamily}
{\bfseries}{\itshape}

\renewcommand\ge{\geqslant} \renewcommand\le{\leqslant}
\let\tildeaccent=\~ \let\hataccent=\^
\renewcommand\~[1]{\widetilde{#1}}

\def\<{\left<} \def\>{\right>} \def\({\left(} \def\){\right)}

\let\polishL=l \def\Zoladek.{\.Zol\c adek}

\def\etc.{\emph{etc}.}

\def\:{\colon}     \def\Q{{\mathbb Q}}

\let\PolishL=\L 
\def\L{{\mathbb L}}

 \def\e{\varepsilon}

 \def\Lojas.{\PolishL ojasiewicz}

\def\NN{\mathbb{N}} \def\ZZ{\mathbb{Z}} \def\QQ{\mathbb{Q}}
  
 \def\PP{\mathbb{P}}

\author
{Raf Cluckers}
\address{Raf Cluckers, Univ.~Lille, CNRS, UMR 8524 - Laboratoire Paul Painlev\'e, F-59000 Lille, France, and,
KU Leuven, Department of Mathematics, B-3001 Leu\-ven, Bel\-gium}
\email{Raf.Cluckers@univ-lille.fr}
\urladdr{https://rcluckers.perso.math.cnrs.fr/}

\author
{Itay Glazer}
\address{Itay Glazer, Department of Mathematics, University of Oxford, Andrew Wiles Building, Radcliffe Observatory
Quarter (550), Woodstock Road, Oxford, OX2 6GG, UK}
\email{itayglazer@gmail.com}
\urladdr{https://sites.google.com/view/itay-glazer}

\thanks{
The authors would like to thank Wouter Castryck, Pierre D\`ebes and Floris Vermeulen for interesting discussions related to the lower bounds of this note.
The author R.\ C. was partially supported by KU Leuven IF C16/23/010 and the Labex CEMPI (ANR-11-LABX-0007-01).}

\title[Eventual tightness of projective dimension growth]{Eventual tightness of projective dimension growth bounds: quadratic in the degree}

\begin{document}


\begin{abstract}
In projective dimension growth results, one bounds the number of rational points of height at most $H$ on an irreducible hypersurface in $\mathbb P^n$ of degree $d>3$ by $C(n)d^2 H^{n-1}(\log H)^{M(n)}$, where the quadratic dependence in $d$ has been recently obtained by Binyamini, Cluckers and Kato in 2024 \cite{BinCluKat}. For these bounds, it was already shown by Castryck, Cluckers, Dittmann and Nguyen in 2020 \cite{CCDN-dgc} that one cannot do better than a linear dependence in $d$. In this paper we show that, for the mentioned projective dimension growth bounds, the quadratic dependence in $d$ is eventually tight when $n$ grows. More precisely the upper bounds cannot be better than $c(n)d^{2-2/n} H^{n-1}$ in general.
Note that for affine dimension growth (for affine hypersurfaces of degree $d$, satisfying some extra conditions), the dependence on $d$ is also quadratic by \cite{BinCluKat}, which is already known to be optimal by \cite{CCDN-dgc}. Our projective case thus complements the picture of tightness for dimension growth bounds for hypersurfaces.
\end{abstract}
\maketitle


In this short paper we prove the eventual tightness of the projective dimension growth upper bounds for hypersurfaces of Theorem 
4 of \cite{BinCluKat}, see Theorem \ref{cor:lowerb} and Proposition \ref{prop:lowerb:en} below. Here, the word eventual indicates that the dimension of the projective space needs to grow.


Let us first recall the projective dimension growth bounds, with both the current state of the art for upper bounds (quadratic in $d$, from Theorem 
4 of \cite{BinCluKat})\footnote{The upper bounds from (\ref{eq:dgu}) are complemented in the case $d=3$ by \cite{Salb:d3,Salberger-dgc} and in the case $d=2$ by \cite{Heath-Brown-Ann}. Note that there are variants to (\ref{eq:dgu}) with upper bounds without any logarithmic factors when $d>4$ (but with a higher power of $d$) in 
 \cite{CCDN-dgc,CHNV-dgc,Pared-Sas,Vermeulen:p}. These variants also treat more general global fields than $\QQ$ and the situation with $X$ of general codimension.} 
 and lower bounds for the worst case (linear in $d$, from \cite[Section 6]{CCDN-dgc}).
%

For $X\subset\PP_{\QQ}^{n}$ and any $H\ge1$, write $X(\QQ,H)$ for the set of rational points $x$ on $X$ that can be written with integer homogeneous coordinates $(x_0,\ldots,x_n)$ where the integers $x_i$ satisfy $|x_i| \le H$ for all $i$.
\begin{Thm}[Projective dimension growth, \cite{BinCluKat} and \cite{CCDN-dgc}]\label{thm:dg}
Let $n>1$ be given. There are constants $c=c(n)$ and $\kappa=\kappa(n)$ such that for any $d\ge4$, any integral hypersurface $X$ of degree $d$ in $\PP_{\QQ}^{n}$ and
any $H>2$ one has
\begin{equation}\label{eq:dgu}
    \# X(\QQ,H) \le c d^{2}H^{\dim X}(\log H)^\kappa.
  \end{equation}
Furthermore, there is a constant $c'=c'(n)>0$ such that for arbitrarily large values of $d$ and $H$ there exists a geometrically integral hypersurface $X$ of degree $d$ in $\PP_{\QQ}^{n}$ such that
  \begin{equation}\label{eq:dgl}
 c' d H^{\dim X} \le  \# X(\QQ,H).
  \end{equation}
\end{Thm}

We give improved lower bounds for worst case projective dimension growth, which are almost
quadratic in $d$ 
instead of linear as in (\ref{eq:dgl}); it is our main result.

\begin{Thm}[Eventual tightness of projective dimension growth]\label{cor:lowerb}
For any $\e>0$ and $N>1$ there exist $n=n(\e)$, $H>N$, $d>N$
and an integral hypersurface $X\subset\PP_{\QQ}^{n}$ of degree $d$
such that
$$
d^{2-\e } H^{\dim X} \le   \# X(\Q,H). 
$$
\end{Thm}

Theorem \ref{cor:lowerb} implies the tightness of the bounds (\ref{eq:dgu}),
up to the logarithmic factors and up to a factor $d^{\e}$. In Proposition \ref{prop:lowerb:en}, we make explicit how $n$ relates to $\varepsilon$, for which we discuss optimality in Remark \ref{rem:optimal}.

\begin{Prop}[Lower bounds for worst case projective dimension growth]\label{prop:lowerb:en}
For any $n\ge2$ and any $N$, there exist $c=c(n)>0$, $H>N$, $d>N$ and an integral hypersurface $X\subset\PP_{\QQ}^{n}$ of degree $d$ such that
\begin{equation}\label{bounds:lower}
cd^{2-2/n}H^{\dim X}\le\#X(\QQ,H).
\end{equation}
\end{Prop}
Note that in any fixed dimension there is still room for a small improvement in the exponent of $d$, either in the lower bounds of Proposition \ref{prop:lowerb:en} or the upper bounds (\ref{eq:dgu}) or both, and we indicate a first step in this direction for curves in Proposition \ref{prop:curves} below.

The dimension growth upper bounds first came as a question by Heath-Brown \cite[p.~227]{Heath-Brown-cubic} and by Serre
\cite[page 178]{Serre-Mordell}, \cite[p.~27]{Serre-Galois}, and addressed the dependence on $H$, with the correct exponent of $H$ being the dimension. The control on the degree $d$ (quadratic in $d$ for hypersurfaces) came more recently, in \cite{BCN-d} for curves over $\QQ$ (after a question by Salberger \cite{Salberger-dgc})  and in \cite{BinCluKat} for general dimension and over any global field; see also \cite{CCDN-dgc,CHNV-dgc,Pared-Sas} for arbitrary codimension with upper bounds which are still polynomial in $d$.

\bigskip
The proofs of Theorem \ref{cor:lowerb} and Proposition \ref{prop:lowerb:en} are based on the key Lemma 
\ref{lem:lowerbound} below, which should be compared with Lemma 6.1 of \cite{CCDN-dgc} and its use in \cite[Section 6]{CCDN-dgc} for the linear bounds from (\ref{eq:dgl}). We discuss tightness for any  $n\ge 2$ in Remark \ref{rem:optimal}.
For dimension growth for \emph{affine} hypersurfaces of degree $d$ (satisfying some extra conditions related to irreducibility), tightness of the quadratic dependence on $d$ is already known by \cite[Theorems 3, 5]{BinCluKat}, \cite[Section 6]{CCDN-dgc}. Our projective case thus complements the picture of tightness for dimension growth bounds for hypersurfaces.
For simplicity we work with $\QQ$-rational points and varieties over $\QQ$ in this paper, but the generalization of Theorem \ref{cor:lowerb} to any global field $K$ (of any characteristic) is possible, corresponding to the set-up of Theorem 4 of \cite{BinCluKat}.

\bigskip
While the upper bounds (\ref{eq:dgu}) of dimension growth rely on the global determinant method from \cite{Salberger-dgc} for points with low multiplicity reduction modulo some primes and a separate treatment for points with high multiplicity reductions in \cite{BinCluKat}, our lower bounds from Theorem \ref{cor:lowerb} and Proposition \ref{prop:lowerb:en} are more elementary and rely on the following key lemma.



\begin{Lem}[Key Lemma]\label{lem:lowerbound} 
For any integers $n\ge2$, $D\ge2$ and $H\ge2$ with
\begin{equation}\label{eq:D:H}
H^{n}\le D^{n-1}
\end{equation}
there exists an irreducible homogeneous polynomial
\[
g\in\QQ[x_{0},x_{1},\ldots,x_{n}]
\]
of degree $d$ with $D\le d\le c_{1}D$ such that $g$ vanishes on
at least $c_{2}H^{n+1}$ many 
rational points on $\PP_{\QQ}^{n}$ of height at most $H$ for some
constants $c_{i}=c_{i}(n)>0$ depending only on $n$. 
\end{Lem}

The proof of the lemma consists of a few steps. We first find a potentially non-irreducible degree $d$ homogeneous polynomial $h$ in $n$ variables which vanishes on all points of bounded height $H$, for $H$ as in (\ref{eq:D:H}). We then apply a trick in number theory (Claim 1) followed by Hilbert's {I}rreducibility {T}heorem to produce from it a new $d$-homogeneous polynomial in $n+1$ variables, which is now irreducible, but still vanishes on a positive proportion of rational points of bounded height. To verify the latter we use a counting argument (Claim 2).

\begin{proof}[Proof of Lemma \ref{lem:lowerbound}] 
First observe that there exists a (non-zero) homogeneous polynomial $h$ in $\QQ[x]=\QQ[x_{0},\ldots,x_{n-1}]$ which vanishes at all
rational points in $\PP_{\QQ}^{n-1}$ of height at most $H$, and, whose degree lies between $D$ and $cD$ for some $c=c(n)>1$. Indeed,
there are $\binom{n-1+d}{n-1}$ many monomials of degree $d$ in $n$ variables, there are no more than $3^{n}H^{n}$ many rational points
in $\PP_{\QQ}^{n-1}$ of height at most $H$, and, each such rational point imposes a homogeneous linear equation on the coefficients of
a general homogeneous polynomial $h$ of degree $d$. Since clearly
\[
\binom{n-1+d}{n-1}>(\frac{d}{n})^{n-1}\ge(\frac{D}{n})^{n-1}\ge\frac{H^{n}}{n^{n-1}},
\]
the observation follows.

Now let $h$ be a polynomial as in the above observation. Decompose $h$ into irreducible factors in $\QQ[x]$ and 
denote the resulting collection of irreducible factors by $(g_{\ell})_{\ell\in I}$ (without repetition). For each $0\leq i\leq n-1$, let $h^{(i)}$
be the product of those $g_{\ell}$, for $\ell\in I$, which depend non-trivially on $x_{i}$. By the Pigeonhole Principle, there must
be some $i$, such that the vanishing set of $h^{(i)}$ contains at least $C(n)H^{n}$ many points of height at most $H$ for some $C(n)$ depending only on $n$. Without
loss of generality, we may assume that $i=0$. Up to multiplying $h^{(0)}$ with a new $\QQ$-irreducible (dummy) polynomial which depends non-trivially on $x_{0}$, we may assume that the degree $d$ of $h^{(0)}$ lies between $D$ and $c_{1}D$ for some $c_{1}$ depending only on $n$.

We can recapitulate our construction so far, by stating that $h^{(0)}$ in $\QQ[x]$ is of degree $d$ between $D$ and $c_{1}D$, is reduced,
each $\QQ$-irreducible factor of $h^{(0)}$ depends non-trivially on $x_{0}$, and that $h^{(0)}$ vanishes on at least $C(n)H^{n}$
many points of height at most $H$.

Now we consider the polynomial $F(x,x_{n},t)$ in $\QQ[x,x_{n},t]$, with $x_{n},t$ two new variables, given by
\begin{equation}
F(x_{0},\ldots,x_{n},t):=h^{(0)}(x_{0},x_{1},\ldots,x_{n-1})+t\cdot h^{(0)}(x_{n},x_{1},\ldots,x_{n-1}). \label{eq:F}
\end{equation}
Note that $h^{(0)}(x_{0},x_{1},\ldots,x_{n-1})$ and $h^{(0)}(x_{n},x_{1},\ldots,x_{n-1})$ are coprime polynomials in $\QQ[x,x_{n}]$ (as the divisors of the left hand side all contain the variable $x_{0}$, and the right hand side does not).

We will use two claims to finish the proof of the lemma. \\

\textbf{Claim 1.} The polynomial $F$ from (\ref{eq:F}) is irreducible in $\QQ[x,x_{n},t]$. \\

Assume toward contradiction that $F$ is reducible, and write
\[
F(x,x_{n},t)=h_{1}(x,x_{n},t)\cdot h_{2}(x,x_{n},t),
\]
where $h_{1}$ and $h_{2}$ are non-constant and in $\QQ[x,x_{n},t]$. Clearly $F$ is irreducible as an element in the ring $\QQ(x,x_{n})[t]$, and hence either $h_{1}$ or $h_{2}$ are units in $\QQ(x,x_{n})[t]$. We may assume that $h_{1}$ is a unit in $\QQ(x,x_{n})[t]$,
and so
\[
h_{1}(x,x_{n},t)=\widetilde{h}_{1}(x,x_{n})
\]
for some polynomial $\widetilde{h}_{1}$ in $\QQ[x,x_{n}]$. Write
\[
h_{2}(x,x_{n},t)=\sum_{\ell=0}^{d}t^{\ell}u_{\ell}(x,x_{n}).
\]
We must have $u_{\ell}=0$ for $\ell\geq2$ and thus we can write
\[
h_{1}(x,x_{n},t)\cdot h_{2}(x,x_{n},t)=\widetilde{h}_{1}(x,x_{n})u_{0}(x,x_{n})+t\widetilde{h}_{1}(x,x_{n})u_{1}(x,x_{n})
\]
and find
\[
h^{(0)}(x_{0},x_{1},\dots,x_{n-1})=\widetilde{h}_{1}(x,x_{n})u_{0}(x,x_{n}),
\]
and
\[
h^{(0)}(x_{n},x_{1},\dots,x_{n-1})=\widetilde{h}_{1}(x,x_{n})u_{1}(x,x_{n}).
\]

Recall that $h^{(0)}(x_{0},x_{1},\ldots,x_{n-1})$ and $h^{(0)}(x_{n},x_{1},\ldots,x_{n-1})$ are coprime polynomials in $\QQ[x,x_{n}]$. This implies that $\widetilde{h}_{1}(x,x_{n})$ is a unit in $\QQ[x,x_{n}]$, which is a contradiction. This finishes
the proof of Claim 1. \\

\textbf{Claim 2. }Let $F(x_{0},\ldots,x_{n},t)$ be from (\ref{eq:F}). For every $\alpha\in\QQ$, the polynomial
\[
F(x_{0},\ldots,x_{n},\alpha)
\]
vanishes on at least $c_{2}H^{n+1}$ many rational points in $\PP_{\QQ}^{n}$ of height at most $H$, for some constant $c_{2}$ depending only on $n$. 
\\

Let us explain how the lemma follows from the two above claims. Let $F$ be the polynomial from (\ref{eq:F}). By Claim 1 and Gauss's Lemma, $F$ is irreducible in $\QQ(t)[x_{0},\ldots,x_{n}]$. By Hilbert's Irreducibility Theorem we find a rational number $\alpha$ such that $F(x_{0},\ldots,x_{n},\alpha)$ is irreducible 
in $\QQ[x_{0},\ldots,x_{n}]$, and we then find that
\[
g(x_{0},\ldots,x_{n}):=F(x_{0},\ldots,x_{n},\alpha)
\]
has all the desired properties for Lemma \ref{lem:lowerbound}. \\

It only remains to prove Claim 2 which we do now. The main idea is that the vanishing set of $F(x_{0},\ldots,x_{n},\alpha)$ contains the intersection of the vanishing sets of $h^{(0)}(x_{0},x_{1},\ldots,x_{n-1})$ and $h^{(0)}(x_{n},x_{1},\ldots,x_{n-1})$, which has a fibered product structure with respect to the projection to the coordinates $x_{1},\ldots,x_{n-1}$. Since $h^{(0)}$ vanishes on positive density of rational points of bounded height, by some form of the Pigeonhole Principle we conclude that so does the fibered product.

Recall that any
$x\in\PP_{\QQ}^{n}(\QQ,H)$ can be represented as a tuple $(x_{0},\ldots,x_{n})$ of integers with greatest common divisor equal to $1$, and which satisfy $\left|x_{i}\right|\leq H$ for all $i\in\{0,\dots,n\}$.
For each $m\in\NN$, denote by
\[
S_{H,m}:=\left\{ (x_{0},\ldots,x_{m})\in\ZZ^{m+1}:\mathrm{gcd}(x_{0},\ldots,x_{m})=1\text{ and }\left|x_{i}\right|\leq H,\,\forall i\in\{0,\dots,m\}\right\} .
\]
Since each $x\in\PP_{\QQ}^{n}(\QQ,H)$ has at most two such representations $\pm(x_{0},\ldots,x_{n})$ in $S_{H,n}$, it is enough to show that for every $\alpha\in\QQ$, the polynomial $F(x_{0},\ldots,x_{n},\alpha)$ vanishes on at least $c_{2}H^{n+1}$ points in $S_{H,n}$. We already know that $h^{(0)}(x_{0},x_{1},\ldots,x_{n-1})$ vanishes on a set $T\subseteq S_{H,n-1}$ of cardinality at least $C(n)H^{n}$. Denote by
\[
T_{\widetilde{x}}:=\left\{ x_{0}\in\ZZ:(x_{0},\widetilde{x})\in T\right\} ,
\]
where $\widetilde{x}:=(x_{1},\ldots,x_{n-1})$,  and note that $T=\bigsqcup_{\widetilde{x}\in\ZZ^{n-1}}T_{\widetilde{x}}$, where $\bigsqcup$ indicates that the union is disjoint.

Since the cardinality of the set $T$ is at least $C(n)H^{n}$, there
are at least $\frac{C(n)}{3^{n}}H^{n-1}$ tuples $\widetilde{x}\in\ZZ^{n-1}$
in which $\#T_{\widetilde{x}}>\frac{C(n)}{3^{n}}H$, as otherwise
\begin{align*}
\#T & =\sum_{\widetilde{x}:\#T_{\widetilde{x}}>\frac{C(n)}{3^{n}}H}\#T_{\widetilde{x}}+\sum_{\widetilde{x}:\#T_{\widetilde{x}}\leq\frac{C(n)}{3^{n}}H}\#T_{\widetilde{x}}\\
 & \leq\frac{C(n)}{3^{n}}H^{n-1}\cdot3H+(3H)^{n-1}\cdot\frac{C(n)}{3^{n}}H<C(n)H^{n},
\end{align*}
which yields a contradiction.

Hence, up to making $T$ 
smaller, we may furthermore assume that for each $x$ in $T$, the set $T_{(x_{1},\ldots,x_{n-1})}$ is of size at least $\frac{C(n)}{3^{n}}H$.
Denote by $S\subseteq\ZZ^{n-1}$ the projection of $T$ to the coordinates
$(x_{1},\ldots,x_{n-1})$, and thus, by our assumption we have $\#S\geq\frac{C(n)}{3^{n}}H^{n-1}$.

Fix $\alpha\in\QQ$ and let $Y_{\alpha}\subseteq S_{H,n}$ be the
vanishing set of $F(x_{0},\ldots,x_{n},\alpha)$ in $S_{H,n}$. Since
$Y_{\alpha}$ contain the common vanishing set of $h^{(0)}(x_{0},x_{1},\ldots,x_{n-1})$
and $h^{(0)}(x_{n},x_{1},\ldots,x_{n-1})$ in $S_{H,n}$, the projection
of $Y_{\alpha}$ to the coordinates $(x_{1},\ldots,x_{n-1})$ has
fibers of size at least $\left(\#T_{(x_{1},\ldots,x_{n-1})}\right)^{2}\geq\frac{C(n)^{2}}{3^{2n}}H^{2}$.
Hence,
\[
\#Y_{\alpha}\geq\sum_{\widetilde{x}\in S}\left(\#T_{\widetilde{x}}\right)^{2}\geq\#S\cdot\frac{C(n)^{2}}{3^{2n}}H^{2}\geq\frac{C(n)^{3}}{3^{3n}}H^{n+1},
\]
as required. This finishes the proof of Claim 2 and thus of Lemma \ref{lem:lowerbound}.
\end{proof} 

\begin{proof}[Proof of Theorem \ref{cor:lowerb}]
Let $\e>0$ be given. Take $n$ sufficiently large (given $\e$), and let $H$, $D$, $g$ and $d$ be as in Lemma \ref{lem:lowerbound} with furthermore $d>N$ and $H>N$. Suppose furthermore that
\begin{equation}
D^{n-1}\le 2H^{n}.\label{eq:n:2H}
\end{equation}
Let $X$ in $\PP_{\QQ}^{n}$ be the hypersurface defined by $g$. Then $X$ is as desired (using that $n$ is large, given $\e$). Indeed, one has by Lemma
\ref{lem:lowerbound} and our extra assumption (\ref{eq:n:2H})
\begin{equation}
c_{3}d^{2-2/n}H^{n-1}\le c_{2}H^{n+1}\le\#X(\QQ,H),\label{eq:n:even}
\end{equation}
for some constant $c_{3}>0$ depending only on $n$, and hence, if $2/n<\e$, then we are done by taking $d$ and $H$ large enough to
dominate the constant $c_{3}$.
\end{proof}

\begin{proof}[Proof of Proposition \ref{prop:lowerb:en}]
This follows from (\ref{eq:n:even}) in the proof of Theorem \ref{cor:lowerb}. 
\end{proof}

\begin{Rem}\label{rem:convolutions}
In order to show Theorem \ref{cor:lowerb}, instead of our key Lemma \ref{lem:lowerbound} one could also use more simply a self-convolution like $g(x,y,z):=h(x)+h(y)+h(z)$ of the polynomial $h(x)$ from the first observation in the proof of our key lemma. Then $g(x,y,z)$ vanishes on a positive proportion of rational points of bounded height, and by the more general theory of convolutions from \cite{GH19,GH21,GHb}, and explicitly by \cite[Theorem B(4)]{GH21}, it follows that $g$ is absolutely irreducible, as required. While this construction leads to a slightly higher number of variables than the polynomial $g$ of our key Lemma \ref{lem:lowerbound} and thus to a less sharp form of Proposition \ref{prop:lowerb:en}, it was the initial inspiration to this note. \end{Rem} 

When looking at the concrete lower bounds of (\ref{bounds:lower}) from Proposition \ref{prop:lowerb:en}, it is tempting to try to improve the upper bounds of (\ref{eq:dgu}) by shaving off some small value in the exponent of $d^{2}$, but we don't see how to do this except for the case of curves in $\PP^{2}$. We make this more precise in Remark \ref{rem:optimal} and in Proposition \ref{prop:curves}.

\begin{Rem}\label{rem:optimal}
Since there is still some gap between the upper bounds from (\ref{eq:dgu}) and the lower bounds (\ref{bounds:lower}) from Proposition \ref{prop:lowerb:en}, one may wonder how to improve either of them. Note however the following heuristic. If one improves our key lemma \ref{lem:lowerbound} to find an irreducible polynomial $g$ of substantially lower degree than in the lemma and which vanishes on all rational points of height up to $H$, then its restriction to $x_0=0$ would yield a polynomial $h_0$ which also vanishes on all rational points of height up to $H$, but which is of substantially lower degree than the polynomial $h$ of the beginning of the proof of the key lemma; the existence of such $h_0$ seems counter intuitive. This heuristic, along with Proposition \ref{prop:curves} seem to indicate that future improvements lie with the upper bounds.
\end{Rem}

The next proposition for curves brings the lower bounds from (\ref{bounds:lower}) closer to the upper bounds by improving the upper bounds from (\ref{eq:dgu}), but there is still a small gap in the exponent of $d$ in (\ref{bounds:lower}) versus (\ref{eq:en2}) with $n=2$ that is prone to further improvement.

\begin{Prop}\label{prop:curves}
Let $X$ be an irreducible curve in $\PP^{n}_\QQ$ of degree $d>1$ for some $n>1$ and let $H>2$ be given. 
Then one has
\begin{equation}
\#X(\QQ,H)\le c(n)d^{4/3}H(\log H)^{\kappa}\label{eq:en2}
\end{equation}
for some constant $c(n)$ and absolute constant $\kappa$.
\end{Prop}
\begin{proof}
For showing (\ref{eq:en2}), we may suppose that $\log H<d$. Indeed, if $\log H\ge d$ then we can replace a factor $d$ by $\log H$ in the bounds (2) of Theorem 
1 of \cite{BinCluKat}. Let us thus suppose that $\log H<d$ holds. By (2) of Theorem 
1 of \cite{BinCluKat}, we then know that
\begin{equation}
\#X(\QQ,H)\le c(n)d^{2}(\log H)^{\kappa}\label{eq:enBCK}
\end{equation}
holds with some constant $c(n)$ and absolute constant $\kappa$. We may furthermore assume that
\begin{equation}
d^{2/3}<H.\label{eq:endH}
\end{equation}
Indeed, there are no more than $c'H^{3}$ many possible points of height at most $H$ in $\PP^{2}$, for some constant $c'$. Combining (\ref{eq:enBCK}) and (\ref{eq:endH}) leads to (\ref{eq:en2}).
\end{proof}

Finally let us mention the challenge to understand the optimal dependence on the degree $d$ for dimension growth bounds in more general codimensions than for hypersurfaces; this is known to be at most polynomial in $d$ by \cite{CCDN-dgc,CHNV-dgc,Pared-Sas}, but the optimal exponent of $d$ still needs to be discovered in general codimension.


\bibliographystyle{plain} \bibliography{nrefs}
\end{document}